\documentclass[oneside,english]{amsart}
\usepackage[T1]{fontenc}
\usepackage[latin9]{inputenc}
\usepackage{geometry}
\usepackage{amsthm}
\usepackage{amstext}
\usepackage{amssymb}
\usepackage{graphicx}

\makeatletter
\numberwithin{equation}{section}
\numberwithin{figure}{section}
\theoremstyle{plain}
\newtheorem{thm}{\protect\theoremname}
  \theoremstyle{remark}
  \newtheorem{rem}[thm]{\protect\remarkname}
  \theoremstyle{definition}
  \newtheorem{defn}[thm]{\protect\definitionname}
  \theoremstyle{plain}
  \newtheorem{prop}[thm]{\protect\propositionname}
  \theoremstyle{plain}
  \newtheorem{cor}[thm]{\protect\corollaryname}
  \theoremstyle{plain}
  \newtheorem{lem}[thm]{\protect\lemmaname}
  \theoremstyle{plain}
  \newtheorem{conjecture}[thm]{\protect\conjecturename}

\usepackage{shuffle}

\makeatother

  \providecommand{\conjecturename}{Conjecture}
  \providecommand{\corollaryname}{Corollary}
  \providecommand{\definitionname}{Definition}
  \providecommand{\lemmaname}{Lemma}
  \providecommand{\propositionname}{Proposition}
  \providecommand{\remarkname}{Remark}
\providecommand{\theoremname}{Theorem}

\begin{document}
\address[Minoru Hirose]{Faculty of mathematics, Kyushu University,
744, Motooka, Nishi-ku, Fukuoka, 819-0395, Japan}
\email{m-hirose@math.kyushu-u.ac.jp}

\title{Double shuffle relations for refined symmetric multiple zeta values}

\author{Minoru Hirose}
\begin{abstract}
Symmetric multiple zeta values (SMZVs) are elements in the ring of
all multiple zeta values modulo the ideal generated by $\zeta(2)$
introduced by Kaneko-Zagier as counterparts of finite multiple zeta
values. It is known that symmetric multiple zeta values satisfy double
shuffle relations and duality relations. In this paper, we construct
certain lifts of SMZVs which live in the ring generated by all multiple
zeta values and $2\pi i$ as certain iterated integrals on $\mathbb{P}^{1}\setminus\{0,1,\infty\}$
along a certain closed path. We call this lifted values as refined
symmetric multiple zeta values (RSMZVs). We show double shuffle relations
and duality relations for RSMZVs. These relations are refinements
of the double shuffle relations and the duality relations of SMZVs.
Furthermore we compare RSMZVs to other variants of lifts of SMZVs.
Especially, we prove that RSMZVs coincide with Bachmann-Takeyama-Tasaka's
$\xi$-values.
\end{abstract}

\date{\today}

\maketitle

\section{Introduction}

For an index $\Bbbk=(k_{1},\dots,k_{d})$, a multiple zeta value $\zeta(\Bbbk)$
is a real number defined by
\[
\zeta(\Bbbk)=\sum_{0<m_{1}<\cdots<m_{d}}m_{1}^{-k_{1}}\cdots m_{d}^{-k_{d}}
\]
where $k_{1},\dots,k_{d-1}\in\mathbb{Z}_{\geq1}$ and $k_{d}\in\mathbb{Z}_{\geq2}$.
Let $\mathcal{Z}$ be a $\mathbb{Q}$-subalgebra of $\mathbb{R}$
generated by $1$ and all multiple zeta values. A symmetric multiple
zeta value (SMZV) is an element of $\mathcal{Z}/\pi^{2}\mathcal{Z}$
defined by
\[
\zeta^{S}(k_{1},\dots,k_{d}):=\left(\sum_{i=0}^{d}(-1)^{k_{i+1}+\cdots+k_{d}}\zeta_{\shuffle}(k_{1},\dots,k_{i})\zeta_{\shuffle}(k_{d},\dots,k_{i+1})\mod\pi^{2}\right)
\]
where $\zeta_{\shuffle}(\Bbbk)$ is the shuffle regularized multiple
zeta value. SMZVs are introduced by Kaneko and Zagier as counterparts
of finite multiple zeta values \cite{KanekoFMZV}\cite{KZ}. In this
paper, we define \emph{refined symmetric multiple zeta values} (RSMZV)
$\zeta^{RS}(\Bbbk)\in\mathcal{Z}[2\pi i]=\mathcal{Z}\oplus2\pi i\mathcal{Z}$
by considering iterated integrals along the non-trivial simple path
from $0$ to $0$ on $\mathbb{P}^{1}(\mathbb{C})\setminus\{0,1,\infty\}$
(see Figure \ref{fig:path_beta}), and show the following properties:
\begin{itemize}
\item $\zeta^{RS}(\Bbbk)$ is a lift of $\zeta^{S}(\Bbbk)$ i.e., $\rho(\zeta^{RS}(\Bbbk))=\zeta^{S}(\Bbbk)$
where $\rho:\mathcal{Z}[2\pi i]\xrightarrow{{\rm Re}}\mathcal{Z}\to\mathcal{Z}/\pi^{2}\mathcal{Z}$
(Corollary~\ref{cor:RSMZV_is_lift_of_SMZV}).
\item $\zeta^{RS}$ satisfies double shuffle relations, duality relations,
and reversal formula (Theorems~\ref{thm:double_shuffle} and~ \ref{thm:duality_and_reversal}).
\item $\zeta^{RS}$ coincides with Bachmann-Takeyama-Tasaka's $\xi$-value
in \cite[Section 2.3.3]{BTT2017} (see Remark~\ref{rem:BTT}).\end{itemize}
\begin{rem}
The definition of $\zeta^{RS}(\Bbbk)$ is essentially equivalent to
the case $N=1$ and $\lambda=0$ of \cite[Definition 4.25]{Jaro2}.
Jarossay already showed the lifting property $\rho(\zeta^{RS}(\Bbbk))=\zeta^{S}(\Bbbk)$
(Corollary~\ref{cor:RSMZV_is_lift_of_SMZV}) and the first formula
of Theorems~\ref{thm:double_shuffle} (see the first sentence of
page 33 of \cite{Jaro2}). He also discussed the double shuffle relations
of $\zeta_{\shuffle}^{S}$ and $\zeta_{*}^{S}$ without modulo $\zeta(2)$
(see \cite[Section 4.6.1]{Jaro2}.
\end{rem}
The contents of this paper are as follows. In Section \ref{sec:iterated_integral},
we give a definition of RSMZVs and show some basic facts. In Section
\ref{sec:relations}, we formulate the double shuffle relations, and
prove the duality relation and the reversal formula. In Section \ref{sec:Other_variants},
we consider other variants of SMZVs and compare them to RSMZVs, and
give a proof of the double shuffle relations. In Section \ref{sec:Complementary-results},
we present some complementary results.

\section{\label{sec:iterated_integral}Iterated integral expression of refined
symmetric multiple zeta values}

\subsection{Iterated integral symbols}

Let us introduce some notions concerning (regularized) iterated integrals.
Our basic references are \cite{gilmultiple} and \cite[Section 2]{Gon}.
We define a \emph{tangential base point} $v_{p}$ as a pair of a point
$p\in\mathbb{C}$ and a nonzero tangential vector $v\in T_{p}\mathbb{C}=\mathbb{C}$.
We define a \emph{path} from $v_{p}$ to $w_{q}$ on a subset $M\subset\mathbb{C}$
as a continuous piecewise smooth map $\gamma:[0,1]\to\mathbb{C}$
such that $\gamma(0)=p$, $\gamma'(0)=v$, $\gamma(1)=q$, $\gamma'(1)=-w$
and $\gamma(t)\in M$ for all $0<t<1$. We denote by $\pi_{1}(M,v_{p},w_{q})$
the set of homotopy classes of paths from $v_{p}$ to $w_{q}$ on
$M$. For tangential base points $x,y,z$ and a subset $M\subset\mathbb{C}$,
the composition map
\[
\pi_{1}(M,x,y)\times\pi_{1}(M,y,z)\to\pi_{1}(M,x,z)\ \ ;\ \ (\gamma_{1},\gamma_{2})\mapsto\gamma_{1}\gamma_{2}
\]
and the inverse map
\[
\pi_{1}(M,x,y)\to\pi_{1}(M,y,x)\ \ ;\ \ \gamma\mapsto\gamma^{-1}
\]
are naturally defined.
\begin{defn}
Fix complex numbers $a_{1},\dots,a_{n}\in\mathbb{C}$ and tangential
base points $x,y$. For $\Gamma\in\pi_{1}(\mathbb{C}\setminus\{a_{1},\dots,a_{n}\},x,y)$,
we define $I_{\Gamma}(x;a_{1},\dots,a_{n};y)\in\mathbb{C}$ as follows.
For a representative $\gamma$ of $\Gamma$, define a function $F_{\gamma}:(0,\frac{1}{2})\to\mathbb{C}$
by
\[
F_{\gamma}(t):=\int_{t<t_{1}<\cdots<t_{n}<1-t}\prod_{j=1}^{n}\frac{d\gamma(t_{j})}{\gamma(t_{j})-a_{j}}.
\]
Then there exists complex numbers $c_{0},c_{1},\dots,c_{n}\in\mathbb{C}$
such that
\[
F_{\gamma}(t)=\sum_{k=0}^{n}c_{k}(\log t)^{k}+O(t\log^{n+1}t).
\]
for $t\to0$. It is known that $c_{0},\dots,c_{n}$ do not depend
on the choice of $\gamma$. We define $I_{\Gamma}(x;a_{1},\dots,a_{n};y):=c_{0}$.
(In this notation, the information of tangential base points $x,y$
is redundant, however, we do not omit them to conform to the standard
notations).
\end{defn}
Note that if the index is admissible then this is just a usual iterated
integral i.e., if $p\neq a_{1}$ and $q\neq a_{n}$ then
\[
I_{[\gamma]}(v_{p};a_{1},\dots,a_{n};w_{q})=\int_{0<t_{1}<\cdots<t_{n}<1}\prod_{j=1}^{n}\frac{d\gamma(t_{j})}{\gamma(t_{j})-a_{j}}.
\]

\subsection{Definition and explicit expression of refined symmetric multiple
zeta values}

We define two tangential basepoints $0'$ and $1'$ by
\[
0'=1_{0},\ 1'=(-1)_{1}.
\]
Put $M=\mathbb{C}\setminus\{0,1\}$. Let ${\rm dch}\in\pi_{1}(M,0',1')$
be (the homotopy class represented by) the straight line from $0'$
to $1'$, $\alpha\in\pi_{1}(M,1',1')$ the path from $1'$ to $1'$
which circles $1$ one times counterclockwise, and $\beta={\rm dch}\cdot\alpha\cdot{\rm dch}^{-1}$
the path from $0'$ to $0'$ which circles $1$ one times counterclockwise
(see Figure \ref{fig:path_alpha}, \ref{fig:path_dch} and \ref{fig:path_beta}).

\begin{figure}
\begin{minipage}[t]{0.33\columnwidth}%
\begin{center}
\includegraphics[scale=0.2]{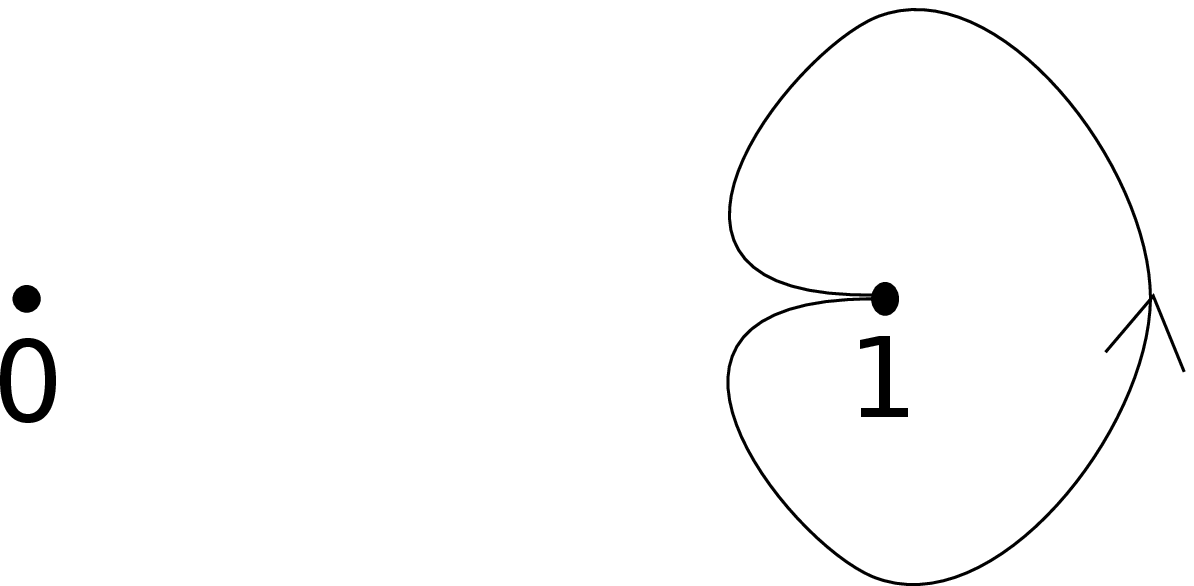} \caption{\label{fig:path_alpha}a path $\alpha$}

\par\end{center}%
\end{minipage}%
\begin{minipage}[t]{0.33\columnwidth}%
\begin{center}
\includegraphics[scale=0.2]{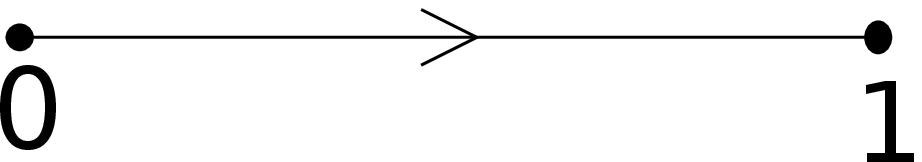}\caption{\label{fig:path_dch}a path ${\rm dch}$}

\par\end{center}%
\end{minipage}%
\begin{minipage}[t]{0.33\columnwidth}%
\begin{center}
\includegraphics[scale=0.2]{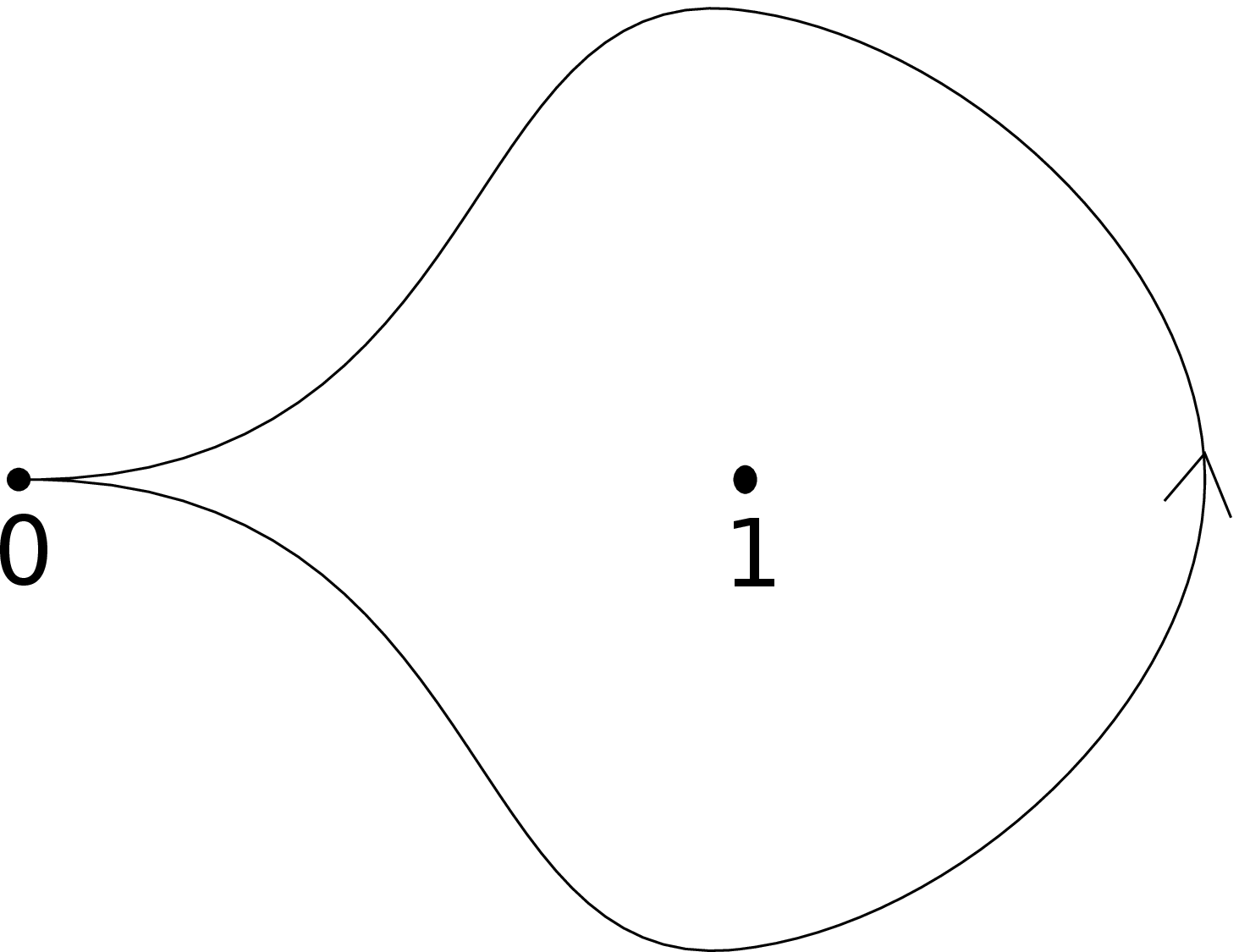}\caption{\label{fig:path_beta}a path $\beta$}

\par\end{center}%
\end{minipage}
\end{figure}
Note that
\begin{align*}
(-1)^{d}\zeta_{\shuffle}(k_{1},\dots,k_{d}) & =I_{{\rm dch}}(0';1,\overbrace{0,\dots,0}^{k_{1}-1},\dots,1,\overbrace{0,\dots,0}^{k_{d}-1};1')\\
 & =(-1)^{k_{1}+\cdots+k_{d}}I_{{\rm dch}^{-1}}(1';\overbrace{0,\dots,0}^{k_{d}-1},1,\dots,\overbrace{0,\dots,0}^{k_{1}-1},1;0')
\end{align*}
and
\[
I_{\alpha}(1';a_{1},\dots,a_{n};1')=\begin{cases}
\frac{(2\pi i)^{n}}{n!} & a_{1}=\cdots=a_{n}=1\\
0 & {\rm otherwise}
\end{cases}
\]
for $a_{1},\dots,a_{n}\in\{0,1\}$ (see \cite[Theorem 3.190 and Examples 3.198]{gilmultiple}).
\begin{defn}
For $d\geq0$ and $k_{1},\dots,k_{d}\in\mathbb{Z}_{\geq1}$, we define
a refined symmetric multiple zeta value $\zeta^{RS}(k_{1},\dots,k_{d})\in\mathbb{C}$
by
\[
\zeta^{RS}(k_{1},\dots,k_{d}):=\frac{(-1)^{d}}{2\pi i}I_{\beta}(0';1,\overbrace{0,\dots,0}^{k_{1}-1},\dots,1,\overbrace{0,\dots,0}^{k_{d}-1},1;0').
\]
For example, $\zeta^{RS}(3,2)=\frac{1}{2\pi i}I(0';1,0,0,1,0,1;0')$
and $\zeta^{RS}(\emptyset)=\frac{1}{2\pi i}I_{\beta}(0';1;0')=1$.
\end{defn}
There are several ways to express RSMZVs by multiple zeta values.
We give one of such expressions obtained by a most naive way here
(see Corollary \ref{cor:RS_and_other_variants} for other expressions). 
\begin{prop}
\label{prop:RS_explicit}We have 
\[
\zeta^{RS}(k_{1},\dots,k_{d})=\sum_{\substack{0\leq a\leq b\leq d\\
k_{j}=1\ \text{for all }a<j\leq b
}
}\frac{(-2\pi i)^{b-a}}{(b-a+1)!}(-1)^{k_{b+1}+\cdots+k_{d}}\zeta_{\shuffle}(k_{1},\dots,k_{a})\zeta_{\shuffle}(k_{d},\dots,k_{b+1}).
\]
\end{prop}
\begin{proof}
Let $n=k_{1}+\cdots+k_{d}+1$ and $(a_{1},\dots,a_{n}):=(1,\overbrace{0,\dots,0}^{k_{1}},1,\dots,1,\overbrace{0,\dots,0}^{k_{d}},1)$.
Then from the path composition formula, we have
\begin{align*}
I_{\beta}(0';a_{1},\dots,a_{n};0')= & \sum_{0\leq l\leq m\leq n}I_{{\rm dch}}(0';a_{1},\dots,a_{l};1')I_{\alpha}(1';a_{l+1},\dots,a_{m};1')I_{{\rm dch}^{-1}}(1';a_{m+1},\dots,a_{n};0')\\
= & \sum_{\substack{0\leq l\leq m\leq n\\
a_{j}=1\ \text{for all }l<j\leq m
}
}\frac{(2\pi i)^{m-l}}{(m-l)!}I_{{\rm dch}}(0';a_{1},\dots,a_{l};1')I_{{\rm dch}^{-1}}(1';a_{m+1},\dots,a_{n};0').
\end{align*}
Here, again from the path composition formula, we have
\[
\sum_{\substack{0\leq l=m\leq n}
}I_{{\rm dch}}(0';a_{1},\dots,a_{l};1')I_{{\rm dch}^{-1}}(1';a_{m+1},\dots,a_{n};0')=I_{{\rm dch}\cdot{\rm dch}^{-1}}(0';a_{1},\dots,a_{n};0')=0.
\]
Thus
\begin{align*}
 & I_{\beta}(0';a_{1},\dots,a_{n};0')\\
= & \sum_{\substack{0\leq l<m\leq n\\
a_{j}=1\ \text{for all }l<j\leq m
}
}\frac{(2\pi i)^{m-l}}{(m-l)!}I_{{\rm dch}}(0';a_{1},\dots,a_{l};1')I_{{\rm dch}^{-1}}(1';a_{m+1},\dots,a_{n};0')\\
= & \sum_{\substack{0\leq a\leq b\leq d\\
k_{j}=1\ \text{for all }a<j\leq b
}
}\frac{(2\pi i)^{b-a+1}}{(b-a+1)!}I_{{\rm dch}}(0';1,\overbrace{0,\dots,0}^{k_{1}-1},1,\dots,1,\overbrace{0,\dots,0}^{k_{a}-1};1')I_{{\rm dch}^{-1}}(1';\overbrace{0,\dots,0}^{k_{b+1}-1},1,\dots,\overbrace{0,\dots,0}^{k_{d}-1}1;0')\\
= & \sum_{\substack{0\leq a\leq b\leq d\\
k_{j}=1\ \text{for all }a<j\leq b
}
}\frac{(2\pi i)^{b-a+1}}{(b-a+1)!}(-1)^{a+d-b+k_{b+1}+\cdots+k_{d}}\zeta_{\shuffle}(k_{1},\dots,k_{a})\zeta_{\shuffle}(k_{d},\dots,k_{b+1})\\
= & (-1)^{d}2\pi i\sum_{\substack{0\leq a\leq b\leq d\\
k_{j}=1\ \text{for all }a<j\leq b
}
}\frac{(-2\pi i)^{b-a}}{(b-a+1)!}(-1)^{k_{b+1}+\cdots+k_{d}}\zeta_{\shuffle}(k_{1},\dots,k_{a})\zeta_{\shuffle}(k_{d},\dots,k_{b+1}).
\end{align*}
Thus the claim is proved.\end{proof}
\begin{cor}
\label{cor:RSMZV_is_in_Z}We have $\zeta^{RS}(k_{1},\dots,k_{d})\in\mathcal{Z}[2\pi i]$.
\end{cor}

\begin{cor}
\label{cor:RSMZV_is_lift_of_SMZV}$\zeta^{RS}(k_{1},\dots,k_{d})$
is a lift of $\zeta^{S}(k_{1},\dots,k_{d})$ i.e., 
\begin{align*}
\zeta^{RS}(k_{1},\dots,k_{d}) & \equiv\sum_{i=0}^{d}(-1)^{k_{i+1}+\cdots+k_{d}}\zeta_{\shuffle}(k_{1},\dots,k_{i})\zeta_{\shuffle}(k_{d},\dots,k_{i+1})\pmod{2\pi i\mathcal{Z}[2\pi i]}.
\end{align*}

\end{cor}

\subsection{Algebraic settings}

Let $\mathbb{Q}\langle e_{0},e_{1}\rangle$ be the free non-commutative
ring generated by formal symbols $e_{0}$ and $e_{1}$ over $\mathbb{Q}$.
Put 
\[
\mathfrak{h}:=e_{0}\mathbb{Q}\langle e_{0},e_{1}\rangle\oplus e_{1}\mathbb{Q}\langle e_{0},e_{1}\rangle\subset\mathbb{Q}\langle e_{0},e_{1}\rangle,
\]
\[
\mathfrak{h}^{0}:=e_{1}\mathbb{Q}\oplus e_{1}\mathbb{Q}\langle e_{0},e_{1}\rangle e_{1}.
\]
For $k_{1},\dots,k_{d}\in\mathbb{Z}_{\geq1}$, define $w(k_{1},\dots,k_{d})\in\mathfrak{h}^{0}$
by
\[
w(k_{1},\dots,k_{d})=(-1)^{d}e_{1}e_{0}^{k_{1}-1}e_{1}\cdots e_{1}e_{0}^{k_{d}-1}e_{1}.
\]
Note that $w(\emptyset)=e_{1}$ where $\emptyset=()$ is an empty
index, and that the elements $w(k_{1},\dots,k_{d})$ with $d\geq0$
and $k_{1},\dots,k_{d}\in\mathbb{Z}_{\geq1}$ form a basis of $\mathfrak{h}^{0}$.
We define a linear map $Z^{RS}:\mathfrak{h}\to\mathbb{C}$ by
\[
Z^{RS}(e_{a_{1}}\cdots e_{a_{k}})=\frac{1}{2\pi i}I_{\beta}(0';a_{1},\dots,a_{k};0').
\]
From the definition, $Z^{RS}(w(k_{1},\dots,k_{d}))=\zeta^{RS}(k_{1},\dots,k_{d})$.

\section{\label{sec:relations}Relations of refined symmetric multiple zeta
values}

\subsection{Double shuffle relations}

We define the shuffle product $\shuffle:\mathbb{Q}\langle e_{0},e_{1}\rangle\times\mathbb{Q}\langle e_{0},e_{1}\rangle\to\mathbb{Q}\langle e_{0},e_{1}\rangle$
by the recursion
\[
u\shuffle1=1\shuffle u=u,
\]
\[
e_{a}u\shuffle e_{b}u'=e_{a}(u\shuffle e_{b}u')+e_{b}(e_{a}u\shuffle u'),
\]
where $a,b\in\{0,1\}$ and $u,u'\in\mathbb{Q}\left\langle e_{0},e_{1}\right\rangle $,
and define the harmonic product $*:\mathfrak{h}^{0}\times\mathfrak{h}^{0}\to\mathfrak{h}^{0}$
by
\[
e_{1}*u=u*e_{1}=u,
\]
\begin{align*}
w(k_{1},\dots,k_{a})*w(l_{1},\dots,l_{b})= & -e_{1}e_{0}^{k_{1}-1}\left(w(k_{2},\dots,k_{a})*w(l_{1},\dots,l_{b})\right)\\
 & -e_{1}e_{0}^{l_{1}-1}\left(w(k_{1},\dots,k_{a})*w(l_{2},\dots,l_{b})\right)\\
 & +e_{1}e_{0}^{k_{1}+l_{1}-1}\left(w(k_{2},\dots,k_{a})*w(l_{2},\dots,l_{b})\right).
\end{align*}
For example, $w(k)*w(l)=w(k,l)+w(l,k)+w(k+l)$.
\begin{thm}[Double shuffle relations for RSMZVs]
\label{thm:double_shuffle}We have
\begin{itemize}
\item $Z^{RS}(u\shuffle v)=2\pi iZ^{RS}(u)Z^{RS}(v)$ for $u,v\in\mathfrak{h}$, 
\item $Z^{RS}(u*v)=Z^{RS}(u)Z^{RS}(v)$ for $u,v\in\mathfrak{h}^{0}$.
\end{itemize}
\end{thm}
The first formula is an immediate consequence of the iterated integral
expression of RSMZVs, but the second one is not obvious from the definition.
We give a proof of Theorem \ref{thm:double_shuffle} in Section \ref{sub:Proof-of-double_shuffle}
in a more general setting.

\subsection{Duality and reversal formula}

We define an automorphism $\varphi$ and anti-automorphism $\tau$
of $\mathbb{Q}\langle e_{0},e_{1}\rangle$ by
\[
\varphi(e_{0})=e_{0}-e_{1},\ \varphi(e_{1})=-e_{1},
\]
\[
\tau(e_{0})=-e_{0},\ \tau(e_{1})=-e_{1}.
\]

\begin{thm}
\label{thm:duality_and_reversal}We have
\begin{itemize}
\item $Z^{RS}(\varphi(w))=-\overline{Z^{RS}(w)}$ for $w\in\mathfrak{h}^{0}$,
\item $Z^{RS}(\tau(w))=-\overline{Z^{RS}(w)}$ for $w\in\mathfrak{h}$.
\end{itemize}
\end{thm}
\begin{proof}
The first formula follows from the M\"{o}bius transformation $t\mapsto\frac{t}{t-1}$.
The second formula follows from the reversal formula of iterated integrals.
\end{proof}

\section{\label{sec:Other_variants}Relation to other versions of symmetric
multiple zeta values}

For $n\in\mathbb{Z}$, put $\beta_{n}={\rm dch}\cdot\alpha^{n}\cdot{\rm dch}^{-1}\in\pi_{1}(\mathbb{C}\setminus\{0,1\},0',0')$.
Define a linear map $L_{n}:\mathbb{Q}\left\langle e_{0},e_{1}\right\rangle \to\mathbb{C}$
by
\[
L_{n}(e_{a_{1}}\cdots e_{a_{k}}):=I_{\beta_{n}}(0';a_{1},\dots,a_{k};0').
\]
For $w\in\mathbb{Q}\left\langle e_{0},e_{1}\right\rangle $, let $L(w;T)\in\mathbb{C}[T]$
be a unique polynomial of $T$ such that
\[
L_{n}(w)=L(w;2\pi in).
\]
We see the existence of such a polynomial as in the proof of Proposition
\ref{prop:RS_explicit} by using the path composition formula. From
the definition, for $w\in\mathfrak{h}$ we have 
\[
Z^{RS}(w)=\frac{1}{2\pi i}L(w;2\pi i).
\]
We can consider many variants of lifts of symmetric multiple zeta
values. In this section, we express such variants by using $L(w;T)$.

\subsection{Generating functions}

For an index $\Bbbk$, we denote by $\zeta_{\shuffle}(\Bbbk;T)\in\mathbb{R}[T]$
(resp. $\zeta_{*}(\Bbbk;T)\in\mathbb{R}[T]$) the shuffle (resp. harmonic)
regularized multiple zeta values with $T$, which are characterized
by the shuffle (resp. harmonic) product identity and $\zeta_{\shuffle}(1;T)=\zeta_{*}(1;T)=T$.
For $d\geq0$ and $k_{1},\dots,k_{d}\in\mathbb{Z}_{\geq1}$, we put
\[
\zeta_{\shuffle}^{S}(k_{1},\dots,k_{d};T_{1},T_{2}):=\sum_{i=0}^{d}(-1)^{k_{d}+\cdots+k_{i+1}}\zeta_{\shuffle}(k_{1},\dots,k_{i};T_{1})\zeta_{\shuffle}(k_{d},\dots,k_{i+1};T_{2}),
\]
\[
\zeta_{*}^{S}(k_{1},\dots,k_{d};T_{1},T_{2}):=\sum_{i=0}^{d}(-1)^{k_{d}+\cdots+k_{i+1}}\zeta_{*}(k_{1},\dots,k_{i};T_{1})\zeta_{*}(k_{d},\dots,k_{i+1};T_{2}).
\]
Let $R=\mathbb{Q}\langle\langle X_{0},X_{1}\rangle\rangle$. Put $\Gamma_{1}(t)=\exp(\sum_{k=2}^{\infty}\frac{\zeta(k)}{k}(-t)^{k})\in\mathbb{R}[[t]]$.
Define an anti-automorphism $\epsilon:R\to R$ by $\epsilon(X_{a})=-X_{a}$.
Put 
\begin{align*}
\Phi_{\shuffle}(T) & :=\sum_{k=0}^{\infty}\sum_{a_{1},\dots,a_{k}\in\{0,1\}}I_{{\rm dch}}(0';a_{1},\dots,a_{k};1')X_{a_{1}}\cdots X_{a_{k}}\exp(-TX_{1})\\
\Phi_{*}(T) & :=\Phi_{\shuffle}(T)\Gamma_{1}(-X_{1})^{-1}\\
\Phi_{\shuffle}^{S}(T_{1},T_{2}) & :=\Phi_{\shuffle}(T_{1})X_{1}\epsilon(\Phi_{\shuffle}(T_{2}))\\
\Phi_{*}^{S}(T_{1},T_{2}) & :=\Phi_{*}(T_{1})X_{1}\epsilon(\Phi_{*}(T_{2}))\\
\Phi^{RS} & :=\sum_{k=1}^{\infty}\sum_{a_{1},\dots,a_{k}\in\{0,1\}}Z^{RS}(e_{a_{1}}\cdots e_{a_{k}})X_{a_{1}}\cdots X_{a_{k}}\\
\Phi^{L}(T) & :=\sum_{k=0}^{\infty}\sum_{a_{1},\dots,a_{k}\in\{0,1\}}L(e_{a_{1}}\cdots e_{a_{k}};T)X_{a_{1}}\cdots X_{a_{k}}.
\end{align*}
We denote by ${\rm coeff}(f,X_{a_{1}}\cdots X_{a_{k}})$ the coefficient
of $X_{a_{1}}\cdots X_{a_{k}}$ in $f$. The following formulas are
essentially proved in \cite[Proposition 10 and Theorem 1]{IKZ}.
\begin{align*}
{\rm coeff}(\Phi_{\shuffle}(T),X_{1}X_{0}^{k_{1}-1}\cdots X_{1}X_{0}^{k_{d}-1}) & =(-1)^{d}\zeta_{\shuffle}(k_{1},\dots,k_{d};T)\\
{\rm coeff}(\Phi_{*}(T),X_{1}X_{0}^{k_{1}-1}\cdots X_{1}X_{0}^{k_{d}-1}) & =(-1)^{d}\zeta_{*}(k_{1},\dots,k_{d};T).
\end{align*}
Thus from the definition, we have
\begin{align*}
{\rm coeff}(\Phi_{\shuffle}^{S}(T_{1},T_{2}),X_{1}X_{0}^{k_{1}-1}\cdots X_{1}X_{0}^{k_{d}-1}X_{1}) & =(-1)^{d}\zeta_{\shuffle}^{S}(k_{1},\dots,k_{d};T_{1},T_{2})\\
{\rm coeff}(\Phi_{*}^{S}(T_{1},T_{2}),X_{1}X_{0}^{k_{1}-1}\cdots X_{1}X_{0}^{k_{d}-1}X_{1}) & =(-1)^{d}\zeta_{*}^{S}(k_{1},\dots,k_{d};T_{1},T_{2}).
\end{align*}

\begin{thm}
We have
\[
\Phi_{\shuffle}^{S}(T_{1},T_{2})=\left.\frac{d}{dT}\Phi^{L}(T)\right|_{T=-T_{1}+T_{2}},
\]
\[
\Phi_{*}^{S}(T_{1},T_{2})=\frac{\Phi^{L}(\pi i-T_{1}+T_{2})-\Phi^{L}(-\pi i-T_{1}+T_{2})}{2\pi i},
\]
\begin{equation}
\Phi^{RS}=\frac{\Phi^{L}(2\pi i)-\Phi(0)}{2\pi i}\ \ \ \ \ (=\frac{\Phi^{L}(2\pi i)-1}{2\pi i}),\label{eq:PhiRS_by_PhiL}
\end{equation}
\[
\Phi^{RS}=\frac{1}{2\pi i}\int_{0}^{2\pi i}\Phi_{\shuffle}^{S}(0,T)dT,
\]
\[
\Phi^{RS}=\Phi_{*}^{S}(-\frac{\pi i}{2},\frac{\pi i}{2}).
\]
\end{thm}
\begin{proof}
It is enough to prove the first and second formulas since the third
formula is obvious from the definition and the fourth and fifth formulas
are consequences of first three formulas. From the path composition
formula, we have
\begin{equation}
\Phi^{L}(T)=\Phi_{\shuffle}(0)\exp(TX_{1})\epsilon(\Phi_{\shuffle}(0)).\label{eq:PhiL_by_PhiShuffle}
\end{equation}
Thus the first and second formulas are proved as follows. For the
first, we compute
\begin{align*}
\Phi_{\shuffle}^{S}(T_{1},T_{2})= & \Phi_{\shuffle}(T_{1})X_{1}\epsilon(\Phi_{\shuffle}(T_{2}))\\
= & \Phi_{\shuffle}(0)\exp(-T_{1}X_{1})X_{1}\epsilon(\Phi_{\shuffle}(0)\exp(-T_{2}X_{1}))\\
= & \Phi_{\shuffle}(0)\exp((-T_{1}+T_{2})X_{1})X_{1}\epsilon(\Phi_{\shuffle}(0))\\
= & \left.\frac{d}{dT}\Phi^{L}(T)\right|_{T=-T_{1}+T_{2}}.
\end{align*}
Here, we have used $\Phi_{\shuffle}(T)=\Phi_{\shuffle}(0)\exp(-TX_{1})$
for the second equality and (\ref{eq:PhiL_by_PhiShuffle}) for the
last equality. For the second one, we compute similarly using \ref{eq:PhiL_by_PhiShuffle}
and the classical formula for the gamma function as 
\begin{align*}
\Phi_{*}^{S}(T_{1},T_{2})= & \Phi_{*}(T_{1})X_{1}\tau(\Phi_{\shuffle}(T_{2}))\\
= & \Phi_{\shuffle}(0)\Gamma_{1}(-X_{1})^{-1}\exp(-T_{1}X_{1})X_{1}\epsilon(\Phi_{\shuffle}(0)\Gamma_{1}(-X_{1})^{-1}\exp(-T_{2}X_{1}))\\
= & \Phi_{\shuffle}(0)\Gamma_{1}(-X_{1})^{-1}\Gamma_{1}(X_{1})^{-1}X_{1}\exp(-(T_{1}-T_{2})X_{1})\epsilon(\Phi_{\shuffle}(0))\\
= & \Phi_{\shuffle}(0)\frac{\sin(\pi X_{1})}{\pi}\exp((-T_{1}+T_{2})X_{1})\epsilon(\Phi_{\shuffle}(0))\\
= & \frac{\Phi^{L}(\pi i-T_{1}+T_{2})-\Phi^{L}(-\pi i-T_{1}+T_{2})}{2\pi i}.
\end{align*}

\end{proof}
Comparing the coefficients of the identities in the theorem, we get
the following corollary.
\begin{cor}
\label{cor:RS_and_other_variants}We have
\[
\zeta_{\shuffle}^{S}((k_{1},\dots,k_{d});T_{1},T_{2})=\left.\frac{d}{dT}L(w(k_{1},\dots,k_{d});T)\right|_{T=-T_{1}+T_{2}},
\]
\begin{equation}
\zeta_{*}^{S}(k_{1},\dots,k_{d};T_{1},T_{2})=\frac{L(w(k_{1},\dots,k_{d});\pi i-T_{1}+T_{2})-L(w(k_{1},\dots,k_{d});-\pi i-T_{1}+T_{2})}{2\pi i},\label{eq:ZetaSh_by_L}
\end{equation}
\[
\zeta^{RS}(k_{1},\dots,k_{d})=\frac{1}{2\pi i}\int_{0}^{2\pi i}\zeta_{\shuffle}^{S}(k_{1},\dots,k_{d};0,T)dT,
\]

\[
\zeta^{RS}(k_{1},\dots,k_{d})=\zeta_{*}^{S}(k_{1},\dots,k_{d};-\frac{\pi i}{2},\frac{\pi i}{2}).
\]
\end{cor}
\begin{rem}
\label{rem:BTT}In \cite{BTT2017}, Bachmann, Takeyama and Tasaka
introduced complex numbers $\xi(\Bbbk)$ as limits of certain finite
multiple harmonic $q$-series. They also prove the equation 
\begin{align*}
\xi(k_{d},\dots,k_{1}) & =\sum_{a=0}^{d}(-1)^{k_{d}+\cdots+k_{a+1}}\zeta_{*}(k_{1},\dots,k_{a};-\frac{\pi i}{2})\zeta_{*}(k_{d},\dots,k_{a+1};\frac{\pi i}{2})\\
 & (=\zeta_{*}^{S}(k_{1},\dots,k_{d};-\frac{\pi i}{2},\frac{\pi i}{2}))
\end{align*}
(see \cite[Theorem 2.6, (2.15)]{BTT2017}). Thus we have
\[
\xi(k_{d},\dots,k_{1})=\zeta^{RS}(k_{1},\dots,k_{d})
\]
from the last formula of Corollary \ref{cor:RS_and_other_variants}.
\end{rem}

\subsection{\label{sub:Proof-of-double_shuffle}Proof of double shuffle relations}
\begin{prop}
\label{prop:L_double}We have
\begin{itemize}
\item $L(u\shuffle v;T)=L(u;T)L(v;T)$ for all $u,v\in\mathbb{Q}\left\langle e_{0},e_{1}\right\rangle $,
\item $\tilde{L}(u*v;T)=\frac{1}{2\pi i}\tilde{L}(u;T)\tilde{L}(v;T)$ for
all $u,v\in\mathfrak{h}^{0}$ where we put $\tilde{L}(u;T):=L(u;T+\pi i)-L(u;T-\pi i)$. 
\end{itemize}
\end{prop}
\begin{proof}
The first formula follows from the shuffle product formula for iterated
integrals. Let $\mathbb{I}$ be the vector space over $\mathbb{Q}$
generated by all indices. We define the harmonic product $*:\mathbb{I}\times\mathbb{I}\to\mathbb{I}$
and the harmonic coproduct $\Delta:\mathbb{I}\to\mathbb{I}\otimes\mathbb{I}$
by
\begin{align*}
(k_{1},\dots,k_{a})*(l_{1},\dots,l_{b})= & (k_{1},(k_{2},\dots,k_{a})*(l_{1},\dots,l_{b}))+(l_{1},(k_{1},\dots,k_{a})*(l_{2},\dots,l_{b}))\\
 & +(k_{1}+l_{1},(k_{2},\dots,k_{a})*(l_{2},\dots,l_{b})),
\end{align*}
\[
\Delta(k_{1},\dots,k_{d}):=\sum_{i=0}^{d}(k_{1},\dots,k_{i})\otimes(k_{i+1},\dots,k_{d}).
\]
Then $(\mathbb{I},*,\Delta)$ is a commutative Hopf algebra. We define
$f:\mathbb{I}\to\mathbb{R}[T_{1},T_{2}]$ by $f(\Bbbk)=\zeta_{*}^{S}(\Bbbk;T_{1},T_{2})$.
Since $\frac{1}{2\pi i}\tilde{L}(w(\Bbbk);-T_{1}+T_{2})=\zeta_{*}^{S}(\Bbbk;T_{1},T_{2})$
from \ref{eq:ZetaSh_by_L} and $w(\Bbbk*\Bbbk')=w(\Bbbk)*w(\Bbbk')$,
the second formula of the proposition is equivalent to
\begin{equation}
f(\Bbbk*\Bbbk')=f(\Bbbk)f(\Bbbk').\label{eq:f_har}
\end{equation}
Then $f$ coincides with the composite map
\[
\mathbb{I}\xrightarrow{\Delta}\mathbb{I}\otimes\mathbb{I}\xrightarrow{g\otimes h}\mathbb{R}[T_{1}]\otimes\mathbb{R}[T_{2}]\xrightarrow{a\otimes b\mapsto ab}\mathbb{R}[T_{1},T_{2}]
\]
where $g,h$ are defined by
\begin{align*}
g(k_{1},\dots,k_{d}) & =\zeta_{*}(k_{1},\dots,k_{d};T_{1}),\\
h(k_{1},\dots,k_{d}) & =(-1)^{k_{d}+\cdots+k_{1}}\zeta_{*}(k_{1},\dots,k_{d};T_{2}).
\end{align*}
Thus (\ref{eq:f_har}) follows from the fact that $(\mathbb{I},*,\Delta)$
is a Hopf algebra and $g,h$ are ring homomorphisms from $(\mathbb{I},*)$
to $\mathbb{R}[T_{1}]$ or $\mathbb{R}[T_{2}]$.
\end{proof}

\begin{proof}[Proof of Theorem \ref{thm:double_shuffle}]
By putting $T=2\pi i$ in the first formula of Proposition \ref{prop:L_double},
we obtain the first formula of Theorem \ref{thm:double_shuffle}.
By putting $T=\pi i$ in the second formula of Proposition \ref{prop:L_double},
we obtain the second formula of Theorem \ref{thm:double_shuffle}.
\end{proof}

\section{\label{sec:Complementary-results}Complementary results and a conjecture}

\subsection{The space generated by RSMZVs}

For an index $\Bbbk=(k_{1},\dots,k_{d})$, we call $d$ the weight
of $\Bbbk$. For $k\in\mathbb{Z}$, we denote by $\mathcal{Z}_{k}$
(resp. $\mathcal{Z}_{k}^{RS}$) the subspace of $\mathbb{C}$ over
$\mathbb{Q}$ generated by all MZVs (resp. RSMZVs) of weight $k$
indices, i.e., we put

\begin{align*}
\mathcal{Z}_{k} & :=\left\langle \zeta_{\shuffle}(k_{1},\dots,k_{d})\mid d\in\mathbb{Z}_{\geq1},\ k_{1},\dots,k_{d}\in\mathbb{Z}_{\geq1},\ k_{1}+\cdots+k_{d}=k\right\rangle _{\mathbb{Q}},\\
\mathcal{Z}_{k}^{RS} & :=\left\langle \zeta^{RS}(k_{1},\dots,k_{d})\mid d\in\mathbb{Z}_{\geq1},\ k_{1},\dots,k_{d}\in\mathbb{Z}_{\geq1},\ k_{1}+\cdots+k_{d}=k\right\rangle _{\mathbb{Q}}.
\end{align*}

\begin{prop}
\label{prop:space_generated_by_RSMZVs}For $k\in\mathbb{Z}$, we have
\[
\mathcal{Z}_{k}^{RS}=\mathcal{Z}_{k}\oplus2\pi i\mathcal{Z}_{k-1}.
\]

\end{prop}
We need some lemma to prove this proposition.
\begin{proof}
Since $\mathcal{Z}_{k}^{RS}\subset\mathcal{Z}_{k}\oplus2\pi i\mathcal{Z}_{k-1}$
from Proposition \ref{prop:RS_explicit}, it is enough to prove that
\[
\mathcal{Z}_{k}\oplus2\pi i\mathcal{Z}_{k-1}\subset\mathcal{Z}_{k}^{RS}.
\]
We prove the proposition by induction on $k$. The case $k<0$ is
trivial. Assume that $\mathcal{Z}_{k-1}^{RS}=\mathcal{Z}_{k-1}\oplus2\pi i\mathcal{Z}_{k-2}$.
Let $x$ be an element of $\mathcal{Z}_{k}$. From the theorem of
Yasuda (\cite[Theorem 6.1]{Yasuda}), there exists $y\in\mathcal{Z}_{k}^{RS}$
such that $x-y\in2\pi i\mathcal{Z}_{k-1}\oplus(2\pi i)^{2}\mathcal{Z}_{k-2}$.
Thus 
\[
\frac{x-y}{2\pi i}\in\mathcal{Z}_{k-1}\oplus2\pi i\mathcal{Z}_{k-2}=\mathcal{Z}_{k-1}^{RS}.
\]
From the special case of the shuffle product formula
\[
Z^{RS}(e_{1}\shuffle u)=2\pi iZ^{RS}(u)\ \ \ (u\in\mathfrak{h}),
\]
we have $x-y\in2\pi i\mathcal{Z}_{k-1}^{RS}\subset\mathcal{Z}_{k}^{RS}$.
Thus $x=(x-y)+y\in\mathcal{Z}_{k}^{RS}$ and the claim is proved.
\end{proof}
It is conjectured by Zagier that
\begin{equation}
\dim_{\mathbb{Q}}\mathcal{Z}_{k}\stackrel{?}{=}d_{k}\label{eq:Zagier_Conj}
\end{equation}
where $(d_{k})_{k}$ are integers defined by
\[
d_{k}=\begin{cases}
0 & k<0\\
1 & k=0\\
d_{k-2}+d_{k-3} & k>0.
\end{cases}
\]
Therefore, from Proposition \ref{prop:space_generated_by_RSMZVs},
the conjectural dimensions of $\mathcal{Z}_{k}^{RS}$ are given by
\[
\dim_{\mathbb{Q}}\mathcal{Z}_{k}^{RS}\stackrel{?}{=}d_{k}+d_{k-1}.
\]

\subsection{Comparison of the shuffle relations of RSMZVs and Kaneko-Zagier's
shuffle relations for SMZVs}

We define $Z^{S}:\mathfrak{h}^{0}\to\mathcal{Z}/\zeta(2)\mathcal{Z}$
by $Z^{S}(w(k_{1},\dots,k_{d})):=\zeta^{S}(k_{1},\dots,k_{d})$. Put
$\tilde{Z}^{S}(u)=Z^{S}(ue_{1})$ and $\bar{u}=-e_{1}\tau(u)e_{1}^{-1}$
for $u\in\mathbb{Q}\oplus e_{1}\mathbb{Q}\left\langle e_{0},e_{1}\right\rangle $.
The following relations of SMZVs were already known.
\begin{prop}[{\cite{KZ},\cite[Corollary 6.1.5]{ZhaoBook}}]
\label{prop:p1}$Z^{S}(u)*Z^{S}(v)=Z^{S}(u)Z^{S}(v)$ for $u,v\in\mathfrak{h}^{0}$.
\end{prop}

\begin{prop}[{\cite{Jaro},\cite{KZ},\cite[Theorem 6.3.4]{ZhaoBook}}]
\label{prop:p2}$\tilde{Z}^{S}(u\shuffle v)=\tilde{Z}^{S}(u\bar{v})$
or equivalently $Z^{S}((u\shuffle v)e_{1})=Z^{S}(ue_{1}\epsilon(v))$
for $u,v\in\mathbb{Q}\oplus e_{1}\mathbb{Q}\left\langle e_{0},e_{1}\right\rangle $.
\end{prop}

\begin{prop}[{\cite[Corollaire 1.12]{Jaro}}]
\label{prop:p3}$Z^{S}(\varphi(u))=-Z^{S}(u)$ for $u\in\mathfrak{h}^{0}$. 
\end{prop}

\begin{prop}[{\cite{KZ},\cite[Theorem 6.3.4]{ZhaoBook}}]
\label{prop:p4}$Z^{S}(\epsilon(u))=-Z^{S}(u)$ for $u\in\mathfrak{h}^{0}$.
\end{prop}
On the other hand, we proved the following relations of RSMZVs in
this paper.
\begin{enumerate}
\item \label{enu:e1}$Z^{RS}(u)*Z^{RS}(v)=Z^{RS}(u)Z^{RS}(v)$ for $u,v\in\mathfrak{h}^{0}$
(Theorem \ref{thm:double_shuffle}). 
\item \label{enu:e2}$Z^{RS}(u\shuffle v)=2\pi iZ^{RS}(u)Z^{RS}(v)$ for
$u,v\in\mathfrak{h}$ (Theorem \ref{thm:double_shuffle}).
\item \label{enu:e3}$Z^{RS}(\varphi(u))=-\overline{Z^{RS}(u)}$ for $u\in\mathfrak{h}^{0}$
(Theorem \ref{thm:duality_and_reversal}). 
\item \label{enu:e4}$Z^{RS}(\epsilon(u))=-\overline{Z^{RS}(u)}$ for $u\in\mathfrak{h}$
(Theorem \ref{thm:duality_and_reversal}). 
\end{enumerate}
Now Propositions \ref{prop:p1}, \ref{prop:p3} and \ref{prop:p4}
are immediate consequences of \ref{enu:e1}, \ref{enu:e3} and \ref{enu:e4}
respectively. However the relation of Proposition \ref{prop:p2} and
\ref{enu:e2} is not obvious. The purpose of this section is to deduce
Proposition \ref{prop:p2} from \ref{enu:e2}. First, we show that
Corollary \ref{cor:RSMZV_is_in_Z} can be extended as follows. 
\begin{lem}
\label{lem:RS_in_Z}For $u\in\mathfrak{h}$, we have
\[
Z^{RS}(u)\in\mathcal{Z}[2\pi i].
\]
\end{lem}
\begin{proof}
The claim follows from the identity
\[
\Phi^{RS}=\Phi_{\shuffle}(0)\frac{\exp(2\pi iX_{1})-1}{2\pi i}\epsilon(\Phi_{\shuffle}(0)),
\]
which is a consequence of (\ref{eq:PhiRS_by_PhiL}) and (\ref{eq:PhiL_by_PhiShuffle}).\end{proof}
\begin{lem}
\label{lem:in_hh}We denote by $\mathfrak{h}\shuffle\mathfrak{h}$
the subspace of $\mathfrak{h}$ spanned by $\{u\shuffle v\mid u,v\in\mathfrak{h}\}$.
Then for all $u,v\in\mathbb{Q}\left\langle e_{0},e_{1}\right\rangle $
and $a\in\{0,1\}$, we have
\[
(u\shuffle v)e_{a}-ue_{a}\epsilon(v)\in\mathfrak{h}\shuffle\mathfrak{h}.
\]
\end{lem}
\begin{proof}
It is enough to consider only the case where $v$ is a monomial $e_{b_{1}}\cdots e_{b_{n}}$.
Then we have
\[
(u\shuffle v)e_{a}-ue_{a}\epsilon(v)=-\sum_{i=1}^{n}(u\shuffle e_{b_{i+1}}\cdots e_{b_{n}})e_{a}\shuffle\epsilon(e_{b_{1}}\cdots e_{b_{i}})\in\mathfrak{h}\shuffle\mathfrak{h}.
\]

\end{proof}

\begin{proof}[Proof of Proposition \ref{prop:p2} using \ref{enu:e2}]
Let $u,v\in\mathbb{Q}\oplus e_{1}\mathbb{Q}\left\langle e_{0},e_{1}\right\rangle $.
From Lemma \ref{lem:in_hh}, we have
\[
(u\shuffle v)e_{1}-ue_{1}\epsilon(v)\in\mathfrak{h}\shuffle\mathfrak{h}.
\]
Hence from \ref{enu:e2} and Lemma \ref{lem:RS_in_Z}, we have
\[
Z^{RS}((u\shuffle v)e_{1}-ue_{1}\epsilon(v))\in2\pi i\mathcal{Z}[2\pi i].
\]
Therefore, from Corollary \ref{cor:RSMZV_is_lift_of_SMZV}, we have
\[
Z^{S}((u\shuffle v)e_{1}-ue_{1}\epsilon(v))=0.
\]
Thus Proposition \ref{prop:p2} is proved.
\end{proof}

\subsection{The strength of the double shuffle relations of RSMZVs}

We define a grading $\mathfrak{h}=\oplus_{k=-1}^{\infty}\mathfrak{h}_{k}$
by $\mathfrak{h}_{k}=\oplus_{a_{1},\dots,a_{k+1}\in\{0,1\}}\mathbb{Q}e_{a_{1}}\cdots e_{a_{k+1}}$.
We put $\mathfrak{h}_{k}^{0}:=\mathfrak{h}_{k}\cap\mathfrak{h}^{0}$.
From Theorem \ref{thm:double_shuffle}, we have
\[
D(u,v,w):=u*(v\shuffle w)-v\shuffle(u*w)\in\ker Z^{RS}
\]
for $u,v,w\in\mathfrak{h}^{0}$. For example
\[
D(e_{1}^{2},e_{1},e_{1})=2w(2)+w(1,1)\in\ker Z^{RS}.
\]

\begin{conjecture}
For $k\in\mathbb{Z}$, the space $\ker Z^{RS}\cap\mathfrak{h}_{k}$
is spanned by
\[
\left\langle D(u,v,w)\mid a+b+c=k-1,\ (u,v,w)\in\mathfrak{h}_{a}\times\mathfrak{h}_{b}\times\mathfrak{h}_{c}\right\rangle .
\]

\end{conjecture}
If we assume Zagier's conjecture (\ref{eq:Zagier_Conj}), the above
conjecture is equivalent to
\begin{equation}
\dim_{\mathbb{Q}}\left\langle D(u,v,w)\mid a+b+c=k-1,\ (u,v,w)\in\mathfrak{h}_{a}\times\mathfrak{h}_{b}\times\mathfrak{h}_{c}\right\rangle _{\mathbb{Q}}\stackrel{?}{=}2^{k-1}-d_{k}-d_{k-1}\label{eq:conj_dim_of_rels}
\end{equation}
for $k\geq1$. We verified (\ref{eq:conj_dim_of_rels}) up to $k=14$
by numerical computation.

\section*{Acknowledgments}

The author would like to thank David Jarossay, Masanobu Kaneko and
Koji Tasaka for useful comments. This work was supported by JSPS KAKENHI
Grant Numbers JP18J00982, JP18K13392.

\bibliographystyle{plain}
\bibliography{RSMZV_ref}

\end{document}